\documentclass[a4paper, 10pt]{article}
\usepackage{mathrsfs}
\usepackage{amsfonts}
\usepackage{latexsym}
\usepackage{graphicx}
\usepackage{amsmath}
\newcommand{\new}{\newcommand*}

\new{\rnew}{\renewcommand*}
\new{\newe}{\newenvironment*}
\new{\newt}{\newtheorem}
\new{\stl}{\setlength}

\new{\bea}{\begin{eqnarray}}
\new{\eea}{\end{eqnarray}}

\new{\be}{\begin{equation}}
\new{\ee}{\end{equation}}

\new{\bean}{\begin{eqnarray*}}
\new{\eean}{\end{eqnarray*}}

\new{\no}{\nonumber}

\new{\bt}{\begin{theorem}}
\new{\et}{\end{theorem}}
\new{\bl}{\begin{lemma}}
\new{\el}{\end{lemma}}

\new{\bc}{\begin{corollary}}
\new{\ec}{\end{corollary}}
\new{\bp}{\begin{proof}\quad}
\new{\ep}{\end{proof}}
\new{\ba}{\begin{array}}
\new{\ea}{\end{array}}

\newt{prop}{Proposition}[section]
\newt{lem}{Lemma}
\newt{conjecture}{Conjecture}[section]
\newt{cor}{Cor}[section]
\newt{thm}{Theorem}
\newt{assumption}{Assumption}
\newt{ex}{Example}
\newt{rem}{Remark}
\newt{definition}{Definition}[section]

\rnew{\theequation}{\thesection.\arabic{equation}}
\new{\sect}[1]{ \section{#1}
\setcounter{equation}{0} \setcounter{figure}{0} }

\newe{acknowledgment}{{\flushleft \bf Acknowledgments. }}{}
\newe{proof}{{\bf Proof.}}{\qquad\endproof}
\newe{keywords}{\begin{quote}\quad{\bf Keywords.}}{\end{quote}}
\newe{AMS}{\begin{quote}\quad {\bf AMS subject classification.}}{\end{quote}}

\def \endproof {\qquad \vrule height 5pt width 5pt depth 2pt}

\title{Beurling's Theorem And Invariant Subspaces\\ For The Shift On Hardy Spaces)\date{}}

\author{{Zhijian Qiu } \\
\\{\small\em  Department of  mathematics,  Southwestern Univ of Finance and Economics}
\\{\small\em  Chengdu 610072,  China, E-mail:qiu@swufe.edu.cn }}

\begin{document}

\maketitle

\begin{abstract}

Let $G$ be a bounded open subset in the complex plane and let
$H^{2}(G)$ denote the Hardy space on $G$. We call a  bounded  simply
connected domain $W$ perfectly connected if the boundary value
function of the inverse of the Riemann map from $W$ onto the unit
disk $D$ is almost 1-1 rwith respect to the Lebesgure on $\partial
D$ and if the Riemann map belongs to the weak-star closure of the
polynomials in $H^{\infty}(W)$. Our main theorem states: In order
that for each $M\in Lat( M_{z})$, there exist $u\in H^{\infty}(G)$
such that $ M = \vee\{u H^{2}(G)\}$, it is necessary and sufficient
that the following hold:

1) Each component of $G$ is a perfectly connected domain.

2)  The harmonic measures of the components of $G$ are mutually
singular.

3) 
  The set of polynomials is weak-star dense in $ H^{\infty}(G)$.

\noindent Moreover, if $G$ satisfies these conditions,
 then  every $M\in Lat( M_{z})$ is of the form $u H^{2}(G)$,
where 
the components of $G$ is either an inner function or zero.
\end{abstract}

\begin{keywords} Hardy space, invariant subspace, shift operator
\end{keywords}
\begin{AMS} 30H05  30E10, 46E15, 47B20
\end{AMS}


\section*{Introduction}
In \cite{beur}  A. Beurling completely described all the invariant
subspaces for the operator "multiplication by z" on the classical
Hardy space $H^{2}(D)$ on the unit disk $D$. $H^{2}(D)$ consists of
analytic functions on $D$ whose coefficients of Taylor expansion are
square summable. A closed linear subspace $M$ of $H^{2}(D)$ is said
z-invariant (or shift invariant) if $ z f \in M $ for every $f\in
M$.

Beurling's theorem states that every z-invariant subspace of
$H^{2}(D)$  is either zero subspace $\{0\}$ or of the form $u
H^{2}(D)$, where $u$ is an inner function in $H^{2}(D)$, that is, a
bounded analytic function on $D$ with nontangential boundary values
of modulus 1 almost everywhere respect to  the Lebesgue measure on
the unit circle $\partial D$.

Beurling's theorem is viewed as one of the most celebrated theorems
in operator theory and it has been extended to many directions.
 Unfortunately, the Beurling type theorem does not hold for the shift
(the operator of multiplication by $z$) on the Hardy spaces on an
arbitrary domain. The lattice of the invariant subspaces of the
shift is unknown in general.  In the case of Bergman space
 on the unit disk, a theorem of Aleman, Richter and Sundberg says that if $M$
is an invariant subspace for the Bergman shift  (which is the
operator of the multiplication by $z$) on the Bergman space
$L^{2}_{a}(G)$, which consists of the analytic functions on $D$ that
are square integrable with respect to the area measure, then $M$ is
generated by $M\ominus z M$, the orthogonal complement of $zM$ in
$M$. This theorem is regarded as the Beurling's theorem for the
Bergman shift. Comparing to the case of the shift on $H^{2}(D)$,
this result does not tell what is $M$.
 Unlike the Hardy space case, the dimension of $M \ominus z
M$ could be any positive integer or $\infty$ (\cite{abfp, hrs}) and
characterization of the invariant subspaces for the Bergman shift is
really difficult.   The lattice of invariant subspaces of the
Bergman shift is unknown.

If  we have our attention back to the Hardy spaces,  then we may
ask: for which  Hardy space  $H^{2}(G)$ it is possible to describe
the lattice of invariant subspaces of the shift operator?

For this question, let us begin with the definition of the Hardy
spaces. For a bounded domain $G$ in the complex plane and a positive
number $q $ in $[1,\infty)$,
 let $ H^{q}(G)$ be the set of analytic
functions $f$ for which there is a function $u$ harmonic on $G$ such
that $|f(z)|^{q} \leq u(z)$ on $G$. Fix a point $a\in G$, then
\[ \|f\| = \inf\{u(a)^{\frac{1}{q}}:
 |f(z)|^{q} \leq u(z),\hspace{.03in} z\in G\hspace{.03in}
\mbox{and}\hspace{.03in}u\hspace{.03in} \mbox{is harmonic on}
\hspace{.03in} G\}
\]
defines a norm on $H^{q}(G)$, which is called   the Hardy space on
$G$ (see \cite{dure}). By the Harnack's inequality, the norms
induced by different points are equivalent.

 If $G$ is a bounded open subsets with
components $G_{1}, G_{2}, ... G_{n}, ... $, let $H^{q}(G)$ denote
the collection  of the  functions to $G$ whose restriction to
$G_{n}$ belongs
 to $H^{q}(G_{n})$ for each $n$. If we define the addition on $H^{q}(G)$
by $ (f+g) (z) = f(z) + g(z) $ pointwise on $G$ and the scalar
multiplication in the obvious way, then $H^{q}(G)$ is a linear space
over the complex field. Moreover, it is easy to verify that
\[\|f\|_{H^{q}(G)}= \{ \sum_{n=1}
 \frac{1}{2^{n}} \|f|G_{n}\|^{q}_{H^{q}(G_{n})} \}^{\frac{1}{q}}\hspace{.1in}
\mbox{for each} \hspace{.1in}f\in H^{q}(G) \] defines a norm on
$H^{q}(G)$ that makes it a complex Banach space of analytic
functions on $G$.

For $q=\infty$, $H^{\infty}(G)$  denotes the Banach algebra of the
bounded analytic functions on $G$ with the supremum norm.

In this paper, we study the invariant space problem for the shift
operator on Hardy space $H^{2}(G)$ over a general open subset set
$G$.
 We investigate the maximum extent of Beurling's theorem in its original
form in Hardy spaces. In other words, we try to find the most
general open subsets
 $G$ in the complex
plane so that the invariant subspace for the shift on $H^{2}(G)$
have as simple forms as those for the shift on $H^{2}(D)$.

The main result of this paper,  Theorem~\ref{t:main}, gives
 a complete description of such open subsets in the plane.
Our characterization is in terms of harmonic measure and perfectly
connected domain as well as weak-star density of polynomials in
$H^{\infty}(G)$.

In Section 1 we give some preliminaries.  In Section 2, we study
perfectly connected domains and the related  weak-star density
problem. In Section 3, we present our main theorem in this paper.

\section{A structure theorem for the  mean closure of the rational functions}
 Let $G$ be a bounded domain with no singleton  boundary component.
 For a fixed $z\in G$, there is a unique probability
measure $\omega_{z}$ supported on $\partial G$ such that
\[u(z) = \int_{\partial G} u(\lambda) d \omega_{z}(\lambda)
\]
for every function $u(\lambda)$ that is continuous on $\overline G$
and harmonic on $G$. The measure $\omega_{z}$ is called the harmonic
measure of $G$ evaluated at $z$. By the Harnack's inequality, any
two harmonic measures for $G$ are boundedly equivalent. The
normalized Lebesgue on $\partial D$ is the harmonic measure
evaluated at 0 for $D$.

For an arbitrary open subset $G$, let $\{G_{i}\}_{i=1}^{\infty}$ be
the collection of  the components of $G$. Then a harmonic measure
for $G$ is defined to be $\omega=\sum_{i=1}^{\infty}
\frac{\omega_{i}}{2^{i}}$, where $\omega_{i}$ is a harmonic measure
for $G_{i}$. The constants $\frac{1}{2^{i}}$, $i=1,2, ...$, are
chosen
 so that $\omega$ become a probability measure on $G$.

Let $K$ be a compact subset in the plane and let $\mu$ be a positive
finite Borel measure supported on $K$. Let $R^{2}(K,\mu)$ denote the
closure of $Rat(K)$, the set of
 the rational functions with poles off $K$, in $ L^{q}(\mu)$.

Let $\cal P$ denote the set of (analytic) polynomials and
$P^{q}(\mu)$ denote the closure of $\cal P$ in $L^{2}(\mu)$.

For $q\in [1,\infty)$, a point $w$ in {\bf C} is called an
 analytic bounded point evaluation
($abpe$) for $R^{q}(K,\mu)$ if there is a neighborhood $G$ of $w$
and a positive constant $c$ such that  for $\lambda \in G$
\[|f(\lambda)| \leq c \hspace{.03in}
 \|f\|_{L^{q}(\mu)},  \hspace{.2in} f \in Rat(K).
\]
So the map, $f \rightarrow f(\lambda)$, extends to a functional in
$R^{q}(K,\mu)^{*}$, the dual space of $R^{q}(K,\mu)$.
 Thus, there is a (kernel) function $k_{\lambda}$
 in $R^{q}(K,\mu)^{*}$ such that
\[f(\lambda) = \int f k_{\lambda} \hspace{.05in}d\mu,\hspace{.1in}
f \in Rat(K).
\]
 Clearly, the set of $abpe$s is open.
For each $f \in R^{q}(K,\mu)$, let $\hat{f}(\lambda) =
 \int f k_{\lambda} \hspace{.05in}d\mu$.
Then $\hat{f}(\lambda)$ is analytic on the set of $abpe$s.

We use $\nabla R^{q}(K,\mu)$ to denote the set of  $abpe$s for
$R^{q}(K,\mu)$.

The proof of our main theorem needs the following structure theorem
for $R^{q}(K,\mu)$, which can be found in \cite{appr}.
Recall that the connectivity of a connected domain  is defined to be
the number of the components of its complement in the complex plane.
A connected open subset is called circular domain if its boundary
consists of a finite number of disjoint circle.

\begin{thm} \label{t:0}
Let $K$ be a compact subset whose complement in the plane has only
finitely many components and let $\mu$ be a positive finite Borel
measure supported on $K$. Then there is a Borel partition $\{
\Delta_{n}\}_{n=0}^{\infty} $ of the support of $\mu$ such that
 \[ R^{q}(K,\mu) = L^{q}(\mu | \Delta_{0})  \oplus R^{q}(K, \mu | \Delta_{1})
 \oplus ... \oplus R^{q}(K, \mu | \Delta_{n}) \oplus ...
\]
and for each $n \geq 1$, if $U_{n} $ denotes $ \nabla R^{q}(K,
\mu|\Delta_{n})$,  then

1)
$\overline{U}_{n} \supset \Delta_{n} $ and
 $R^{q}(K, \mu | \Delta_{n}) = R^{q}(\overline{U}_{n}, \mu | \Delta_{n})$;


2) the map $e$, defined by $e(f)=\hat{f}$, is an isometrical
isomorphism and weak-star homeomorphism from
 $R^{q}(K, \mu|\Delta_{n}) \bigcap L^{\infty}(\mu|\Delta_{n})$
onto $H^{\infty}(U_{n})$.

3)  $U_{n}$ is conformally equivalent to a circular domain and the
connectivity of $U_{n}$ does not exceed the connectivity of the
component of $G$ that contains $U_{n}$;

4) $\mu| \partial U_{n} \ll \omega_{U_{n}}$, the harmonic measure of
$U_{n}$; and if $u_{n}$ is a conformal map from a circular domain
$W_{n}$ onto $U_{n}$, then every $f\in H^{\infty}(U_{n})$ has
nontangential limits on $\partial U_{n}$ {\em a.e.} $[\mu] $ and
\[ e^{-1}(f)(a)=\lim_{z\rightarrow u_{n}^{-1}(a)} f\circ u_{n}(z) \hspace{.05in}
\mbox{\em{for almost every} on} \hspace{.05in}a\in
\partial G\hspace{.02in}\mbox{
with respect to }\hspace{.05in}  \mu|\partial U_{n};
\]

5)         for  each $f\in H^{\infty}(U_{n})$, if let $f^{*}$  be
equal to its nontangential limit values on $\partial U_{n}$  and
$f^{*} = \hat{f}$ on $U_{n}$, then the map $m$, defined by $m(f) =
f^{*}|\Delta_{n}$, is the inverse of the map $e$.
\end{thm}

An observation is that 3) in the theorem above implies that each
$R^{q}(\overline U_{n},\mu|\Delta_{n})$ contains no non-trivial
characteristic function (because $H^{\infty}(U_{n})$ does not).

For a given $\mu$, let $K$ be a closed disk that contains the
support of $\mu$, then $ R^{2}(K,\mu)=P^{2}(\mu)$ and
$R^{2}(K,\mu|\Delta_{i}) = P^{2}(\mu|\Delta_{i})$. So
Theorem~\ref{t:0} is also a structural theorem for $P^{2}(\mu)$. In
1991 J. Thomson first proved Theorem~\ref{t:0} for $P^{2}(\mu)$
except 4) and 5) in \cite{jm}. Some applications of
Theorem~\ref{t:0} can be found in \cite{quas, dens}.

\section{On Perfectly Connected Domains}
Following \cite{glic} we call a simply connected domain $G$ nicely
connected if the boundary value of a conformal map from the unit
disk $D$ onto $G$ is univalent almost everywhere on $\partial D$
with respect the Lebesgue measure. So every bounded analytic
function in $H^{\infty}(G)$
 has a well-defined boundary valuefunction  on $\partial G$ if $G$ is a nicely
connected domain.
 A simply connected domain $G$ is nicely connected
domain if and only if every function in $H^{\infty}(G)$ can be
approximated by a bounded sequence of functions analytic on $G$ and
continuous on $\overline G$ (for example, see \cite{davi}).

For a simply connected domain $G$, recall that the weak-star
topology for $H^{\infty}(G)$ is defined as follows:  let $m$ denote
the normalized Lebesgue measure and let $P^{\infty}(m)$ denote the
weak-star closure of $\cal P$ in  $L^{\infty}(m)$. The map $f
\rightarrow \tilde{f}$ (where $\tilde{f}$ denotes the boundary value
function of $f$ on $\partial D$) from $H^{\infty}(D)$ to
$P^{\infty}(m)$ is an isometrical isomorphism. Since $P^{\infty}(m)$
is the dual space of $L^{1}(m)/P^{\infty}(m)^{\perp}$, we have a
weak-star topology on $H^{\infty}(D)$ that is induced from
$P^{\infty}(m)$ via the isometrical isomorphism. For a bounded
simply connected domain $G$, a conformal map from $D$ onto $G$
induces an isometrical isomorphism from $H^{\infty}(D)$ onto
$H^{\infty}(G)$ in the obvious way. Thus the weak-star topology on
$H^{\infty}(G)$ is defined to be the one induced from
$H^{\infty}(D)$ by that map.

Now we introduce the following:

\vspace{.1in} \noindent {\bf Definition.} A bounded  simply
connected domain $G$ is called perfectly connected if the Riemann
map from $G$ onto $D$ belongs to the weak-star closure of the
polynomials in $H^{\infty}(G)$. \vspace{.1in}

\begin{thm} \label {t:a}
Let $G$ be a bounded simply connected domain and let $\omega$ be
 a harmonic measure for $G$. Then the following are equivalent:

 1) $G$ is a perfectly connected domain.

2) $G$ is a nicely connected domain and the boundary value function
of the Riemann map from $G$ onto $D$ belongs to
$P^{\infty}(\omega).$

3) The set of polynomials is weak-star dense in  $H^{\infty}(G)$.

\end{thm}

\begin{proof}

1) $\Longrightarrow$ 2). Suppose  that $G$ is a  perfectly connected
domain. Let $\psi$ denote the Riemann map from $G$ onto $D$ and let
$\varphi$ denote its inverse function.
 Then $\psi$ can be weak-star approximated by polynomials
in $H^{\infty}(G)$. Without loss of generality, we may assume that
$\omega = m\circ \psi$. Choose a net of polynomials $\{q_{\alpha}\}$
such that it weak-star converges to $\psi$ in $H^{\infty}(G)$. So
the net $\{q_{\alpha}\circ \varphi\}$ weak-star converges to $z$ in
$H^{\infty}(D)$. This is equivalent to that $\{q_{\alpha}\circ
\varphi\}$ weak-star converges to $z$ in $P^{\infty}(m)$, where we
still use $\varphi$ to denote the boundary value of $\varphi$ on
$\partial D$. We claim that $z$ belongs to the closure of $\{p\circ
\varphi: p\in \cal P\}$ in $P^{2}(m)$. In fact, recall that a convex
set is norm closed if and only if it is weakly closed. Since the
weak topology and weak-star topology coincide on a Hilbert space and
since $z$ belongs the weak-star closure of $\{p\circ \varphi: p \in
\cal P\}$ in $P^{\infty}(m)$,
 the claim clearly follows.

To show $G$ is nicely connected, choose a sequence of polynomials
$\{p_{n}\circ\varphi \}$  such that it converges to $z$ in
$P^{2}(m)$. By passing to a subsequence if necessary, we have that
$p_{n}\circ\varphi \rightarrow z$ almost everywhere on $\partial D$
with respect to $m$. This clearly implies that $\varphi$ is
univalent almost everywhere on $\partial D$. Thus, $G$ is nicely
connected. Therefore, every $f$ in $H^{\infty}(G)$ has a boundary
value function, $\tilde f$,  on $\partial G$.

Since $\{q_{\alpha}\circ \varphi\}$ weak-star converges to $z$ in
$P^{\infty}(m)$, we have that
\[ \lim_{\alpha}\int  q_{\alpha}\circ\varphi
 f dm = \int zf dm,\hspace{.2in} f\in L^{1}(m).
\]
So
\[ \lim_{\alpha}\int  q_{\alpha}(f\circ\psi) dm\circ\psi =
\int \psi(f\circ\psi)dm\circ\psi ,\hspace{.2in}
 f\in L^{1}(m).
\]
Since $L^{1}(\omega) = \{f\circ\psi: f\in L^{1}(m)\}$, it follows
that $ \lim_{\alpha} \int q_{\alpha} g d\omega = \int \psi g
d\omega,\hspace{.02in} f\in L^{1}(\omega). $ Therefore, $\psi$ is in
$P^{\infty}(\omega)$.

2) $\Longrightarrow $ 3). With the same notations as above. We
suppose that  $G$ is a nicely connected domain and $\psi$ belongs to
$P^{\infty}(\omega).$

Let $f\in H^{\infty}(G)$, for convenience, we still use $f$ to
denote its boundary value function. Then $f\circ\varphi \in
H^{\infty}(D)$. So there exists a net of polynomials $\{q_{\tau}\}$
such that it weak-star converges  to $f\circ\varphi$ in
$H^{\infty}(G)$. This, in turn, implies  that
$\{q_{\tau}\circ\psi\}$ weak-star converges to $f\in
L^{\infty}(\omega)$. Since $q_{\tau}\circ\psi$ belongs to
$P^{\infty}(\omega)$ for each $n$, it follows that $f\in
P^{\infty}(\omega)$. Hence $H^{\infty}(G) \subset
P^{\infty}(\omega)$. Obviously, $P^{\infty}(\omega )\subset
H^{\infty}(G)$, so we conclude that $P^{\infty}(\omega)
=H^{\infty}(G)$.  Therefore,
 $\cal P$ is weak-star dense in $H^{\infty}(G)$.

3) $\Longrightarrow $ 1) follows by the definition.

\end{proof}

\begin{rem}
There is no simple topological descrition for perfectly connected
domains, however, a Jordan domain (enclosed by a Jordan curve) is
perfectly connected and so is a Carathe\"{o}dory domain (a domains
whose boundary is equal to  the boundary of their complement in the
plane).

With Theorem~\ref{t:a}, we know that a perfectly connected domain is
the image of an bounded analytic function on $D$ known as
 a weak-star generator, which is characterized  by D. Sarason
in \cite{gen}.  Though a perfectly connect domain may look so 'bad'
(see \cite{gen, ord}),
 however, it represents the class of simply connected domains most
resemble to the unit disk from the view of the theory of Hardy
spaces. Our main theorem itself would be an example that reflects
this view.
\end{rem}

Let $1\leq q \leq \infty$ and let $G$ be a nicely connected domain.
Then each function in $H^{q}(G)$ has a well-defined boundary value
function on $\partial G$. For $f\in H^{q}(G)$, let $\tilde f$ denote
the boundary value function on $\partial G$. The map, $f\rightarrow
\tilde f$ is an isometrical isomorphism from $H^{q}(G)$ onto
$\widetilde{H}^{q}(G)$, which is defined to be the subspace
$\{\tilde {f}: f\in H^{q}(G)\}$ in $L^{q}(\omega)$ (where $\omega$
is chosen so that $\|r\|_{H^{q}(G)} = \|r\|_{L^{q}(\omega)}$ for
each $\in Rat(\overline G)$).

Now, for a bounded open subset, if each of the components of $G$ is
nicely connected and the harmonic measures of $G$ are mutually
singular, then every function $f$ in $H^{q}(G)$ also has a well
defined boundary value function $\tilde f$ on $\partial G$. As
before, we use ${\widetilde H}^{q}(G)$ to denote the subspace
$\{\tilde f: f\in H^{q}(G)\}$ in $L^{q}(\omega)$ (again,  $\omega$
is chosen so that $\|r\|_{H^{q}(G)} = \|r\|_{L^{q}(\omega)}$ for
each $r\in Rat(\overline G)$). Moreover, $f\rightarrow \tilde f$ is
an isometrical isomorphism.

In this paper, for seek of convinince, we often don't distinguish
between the two spaces. So, for any compact subset $K$ that contains
$\partial G$, when we say that $R^{q}(K,\omega) = H^{q}(G)$, we
really means that $R^{q}(K,\omega) = \widetilde H^{q}(G)$, provided
that $\omega$ is chosen so that the map $\tilde f \rightarrow f$ is
an isometry (note, in the case that $q=\infty$, any harmonic measure
$\omega$ will do the job). Keep this observation in mind.

\begin{lem} \label{l:conn}
Let $G$ be a bounded domain and let $\omega $ be a harmonic measure
for $G$. If $K$ is a compact subset that contains  $G$, then
 $\nabla R^{2}(K,\omega)$ is  a connected subset that contains $G$.
\end{lem}
\begin{proof}

Without loss of generality, we may assume that $\omega$ is chosen so
that $\|r\|_{L^{2}(\omega}) = \|r\|_{H^{2}(G)}\hspace{.02in}\mbox{
for each } \hspace{.021in}r\in Rat(K). $
 Since $G\subset K$, $Rat(K) \subset H^{2}(G)$. So we can
 define a linear operator
$   J: \hspace{.02in} R^{2}(K,\omega) \rightarrow
H^{2}(G)\hspace{.04in} \mbox{by } \hspace{.021in} J(r)=r
\hspace{.021in} \mbox{ for each }\hspace{.021in} r \in Rat(K). $
Then $ \|J(r)\| = \|r\|_{H^{2}(G)}
 = \|r\|_{L^{2}(\omega)}.
$ So $J$ is an isometry from $R^{2}(K,\omega)$ onto
$J[R^{2}(K,\omega)]$ and the later is a closed subspace of
$H^{2}(G)$. Moreover, it is easy to verify that $J(fg)= J(f)J(g)
\hspace{.02in}\mbox{ for } f,\hspace{.02in} g\in R^{2}(K,\omega)\cap
L^{\infty}(\omega). $ Now let $u\in R^{2}(K,\omega)\cap
L^{\infty}(\omega)$ be a non-zero characteristic function. Then $
[J(u)]^{2} = J(u^{2})=J(u)$. So we have that $J(u-1) = J(u)-1 = 0$,
and hence $u=1$. This means that   $R^{2}(K,\omega)$ is irreducible
and therefore $\nabla R^{2}(K,\omega)$ is connected.

Lastly, the definition of harmonic measure and the Harnack's
inequality implies that $G\subset \nabla R^{2}(K,\omega)$. So the
proof is complete.

\end{proof}

The following result extends Theorem  1 in  \cite{dens}.

\begin{thm} \label {p:dens}
Let $G=\cup G_{i}$ and let $\omega$ be a harmonic meausre for $G$.
Suppose that  each $G_{i}$ is nicely connected and the harmonic
measures of the components of $G$
 are mutually singular. Then
$P^{2}(\omega) = H^{2}(G)$ if and only if $\nabla P^{2}(\omega) =
G$.

\end{thm}

\begin{proof}
 Sufficiency.

Suppose $\nabla P^{2}(\omega) = G$. By Lemma~\ref{l:conn}, we have
that
 $G_{i} \subset \nabla P^{2}(\omega_{i})$. This together with our hypothesis
implies that   $G_{i} = \nabla P^{2}(\omega_{i})$. So it follows by
Theorem 1 in \cite{dens} that $P^{2}(\omega|\partial G_{i} )=
H^{2}(G_{i})$. According to 5) of  Theorem~\ref{t:0}, we have that
$\chi_{\partial G_{i}} $, the characteristic function of $G_{i}$, is
in $P^{2}(\omega)$ for each $i\geq 1$. Thus, we conclude that
\begin{align*}
P^{2}(\omega) & =P^{2}(\omega|\partial G_{i})\oplus ... \oplus
P^{2}(\omega|\partial G_{i}) \oplus ... \\& =H^{2}(G_{1}) \oplus ...
\oplus H^{2}(G_{i})\oplus ... \\& =H^{2}(G).
\end{align*}

Necessity.

The proof is the same as that in the proof of Theorem 1 in
\cite{dens}.

\end{proof}

Recall that a function $g$ in $H^{\infty}(D)$ is called inner if its
boundary values has  modulus 1 almost everywhere on $\partial D$. An
analytic function on $D$ is called  outer if it is of the form
\[f(z)= \alpha \exp\{\frac{1}{2\pi} \int^{\pi}_{-\pi} \frac{
e^{it}+z}{e^{it}-z} \gamma(t) d t\},
\]
where $\alpha$ is a constant with modulus 1 and $\gamma$ is a
real-valued function in $L^{1}(m)$.

The classical Hardy space theory guarantees that every function in
$H^{2}(D)$ can be expressed as a product of an inner and outer
functions. The expression is unique up to a constant with modulus 1
\cite{hoff}.

Let $\psi$ is the Riemann map from $G$ onto $D$.  We call $g\in
H^{2}(G)$ an outer function if $g\circ \psi^{-1}$ is outer in
$H^{2}(D)$. The Beurling's theorem implies that $\phi H^{2}(D)$ is
dense in $H^{2}(D)$ if $\phi$ is outer in $H^{2}(D)$.

The following theorem together with Theorem~\ref{t:a} also gives a
characterization for perfectly connected domain.

\begin{thm} \label {t:out}
Let $G=\cup G_{i} $ and let $\omega$ be a harmonic meausre for $G$.
 Suppose that  each $G_{i}$ is nicely connected and
the harmonic measures of the components of $G$
 are mutually singular. Then
$P^{\infty}(\omega) = H^{\infty}(G)$ if and only if  \hspace{.015in}
$\cal P g$ is dense in $H^{2}(G)$ for each $g$ in $H^{2}(G)$ for
which $g\chi_{G_{i}}$ is outer in $H^{2}(G_{i})$, $i=1, 2, ...$
\end{thm}

\begin{proof}
Sufficiency.

Suppose that $\cal P g$ is dense in $H^{2}(G)$ for each $g$ in
$H^{2}(G)$ for which $g\chi_{G_{i}}$ is outer in $H^{2}(G_{i})$. We
claim that
\[ \nabla P^{2}(\nu)\subset G\hspace{.1in}\mbox{if
$\nu$ is a finite measure such that}\hspace{.1in} [\nu]=[\omega].
\]
Note, if we take g=1, then we see that $\cal P$ is dense in
$H^{2}(G)$.  So it follows by  Proposition~\ref{p:dens} that $\nabla
P^{2}(\omega) = G$.

Set $y=\frac{d\nu}{d\omega} +1$. Then both $y$ and $\log y$ belong
to $ L^{1}(\omega)$. It follows by Szeg\"{o}'s theorem that there
exist an outer function $x_{i}$ in $H^{2}(G_{i})$ such that
$|x_{i}|^{2}=y\chi_{G_{i}}$  for each $i\geq 1$. Clearly,  for each
positive integer $k$, $|\sum_{i=1}^{k} x_{i}|^{2}= \sum_{i=1}^{k}
|x_{i}|^{2} \leq |y|. $
 It follows by the dominated theorem that
$\lim_{k\rightarrow \infty} \sum_{i=1}^{k} x_{i}$ in $H^{2}(G)$
exists.
 Set $x=\sum_{i=1}^{\infty} x_{i}$.
Then $x$ belongs to  $H^{2}(G)$ and  $|x|^{2} = y $. Define a linear
operator $A$ from $P^{2}(y\omega)$ to $P^{2}(\omega)$ densely by
$A(p) = xp\hspace{.1in}\mbox{ for each} \hspace{.1in} p\in \cal P. $
Then
\[ \int |A(p)|^{2} d \omega = \int |xp|^{2} d\omega
= \int|p|^{2} d (y\omega).
\]
Thus, $A$ is an isometry. By our hypothesis, $A$ has a dense range.
So it follows that $A$ is an unitary operator. For any $f\in
P^{2}(y\omega)$, there is $\{p_{n}\}\subset \cal P$ such that $\lim
p_{n}= f$ in $P^{2}(y\omega)$. So
\[\lim A(p) = \lim p_{n} x = fx = A(f)\hspace{.1in}\mbox{ in}\hspace{.1in}
 P^{2}(\omega)
\]
and hence $\widehat{fx}$ is analytic on $G$. Since $\hat{x} \neq  0$
on $G$, we see that $f$ extends to be analyitc on $G$. This implies
that $G \subset \nabla P^{2}(y\omega)$. Similarly, we can show that
$\nabla P^{2}(y\omega) \subset G$. Thus, we conclude that
\[ \nabla P^{2}(\nu) \subset \nabla P^{2}(y\omega)
= \nabla P^{2}(\omega)=G.
\]
This proves  the claim. It is a well-known result (see \cite{jm,
mc}) that there is a positive finite measure $\eta$ such that
$[\eta]=[\omega]$ and $P^{\infty}(\omega)=P^{2}(\eta)\cap
L^{\infty}(\eta)$. So it follows that
\begin{align*}
H^{2}(G) &\supset P^{2}(\omega) \\& \supset P^{\infty}(\omega) \\& =
P^{2}(\eta)\cap L^{\infty}(\eta)\\& =H^{\infty}(\nabla P^{2}(\eta)).
\end{align*}
Consequently, we have  that $\nabla P^{2}(\eta) = G$. Hence, it
follows by the inequality above that
$P^{\infty}(\omega)= H^{\infty}(G)$. So we are done.

Necessity.

Suppose $P^{\infty}(\omega) = H^{\infty}(G)$. Let $g$ be a function
in $H^{2}(G)$ such that $g\chi_{G_{i}}$ is outer in $H^{2}(G_{i})$
for each $i\geq 1$.
 Pick an arbitrary functin  $f\in H^{2}(G)$. For given $\epsilon > 0$,
and each  $i\geq 1$, since $g_{i}$ is outer,  there is a function
$h_{i}$ in $H^{\infty}(G_{i})$ such that
\[\|f\chi_{\partial G_{i}}- g\chi_{\partial G_{i}}  h_{i}\| <
\frac{1}{2}(\frac{\epsilon}{2^{i}}).
 \]
Thus, when $k$ is sufficiently large, we have
\begin{align*}
\|f-g\sum_{i=1}^{k}h_{i}\|&\leq\|\sum_{i=1}^{k}(f\chi_{G_{i}}-g\chi_{G_{i}}
h_{i})\| + \|\sum_{i=k+1}^{\infty} f\chi_{G_{s}}\| \\& \leq
\sum_{i=1}^{k} \frac{1}{2}
(\frac{\epsilon}{2^{i}})+\frac{\epsilon}{2}= \epsilon .
\end{align*}
Set $h=\sum _{i=1}^{k} h_{i}$, then $h\in H^{\infty}(G)$ and
$\|f-gh\|< \epsilon$. Hence, $H^{\infty}(G)g$ is dense in
$H^{2}(G)$, Finally, since
\[H^{\infty}(G) =P^{\infty}(\omega)\subset P^{2}(\omega),
\]
it follows clearly that $\cal P g$ is dense in $H^{2}(G)$. `
\end{proof}

\section{The main result}

For a bounded domain $G$, a point $a$ in $\partial G$ is said to be
removable for $H^{2}(G)$ if every function in $H^{2}(G)$ extends
analyticly to a
 neighborhood of $a$.  Every isolated point in $\partial G$ is removable
for $H^{2}(G)$ (see \cite{axler}).

For a bounded linear operator $T$, let $Lat (T)  $  denote the
lattice of invariant subspaces of $T$.

Let $X$ be a linear topological space and let $F$ be a subset of
$X$. We use $\vee\{F \}$ to denote the closure of the  linear span
of $F$ in $X$.

\begin{lem} \label{l:main}
Let $G$ be a bounded domain that has no removable boundary points
for $H^{2}(G)$.  Let $M_{z}$ denote the shift on $H^{2}(G)$. In
order for that every $M \in Lat( M_{z})$ there exist a function
$u\in H^{\infty}(G)$ such that $M=\vee \{uH^{2}(G)\}$ it is
necessary and sufficient that $G$ is a perfect connected domain.
\end{lem}

\begin{proof}
Necessity.

Let $K$ be the union of $\overline G$ with the bounded components of
the complement of $\overline G$ in the plane and  let $\omega$
denote a harmonic measure of $G$.
 Define $I$: $R^{2}(K,\omega) \rightarrow H^{2}(G)$
by $I(r)= r $, for each $r\in Rat(K)$. Then $I$ is continuous and
bounded below. Then $N=RI(^{2}(K,\omega))$ is the closed subspace
generated by $Rat(K) $ in $H^{2}(G)$.

We first show that $G$ must be simply connected. Suppose that $G$ is
not simply connected, then the complement of $G$ contains a
component $E$ that is a compact subset.
Let $W$ denote $\nabla R^{2}(K,\omega)$. Since the  components of
$K$ are simply connected, it follows by Lemma~\ref{l:conn} and
Theorem~\ref{t:0} that $W$ is a simply connected domain that
contains $G$. Moreover, $W$ contains $E$. Since   isolated points in
$\partial G$ are removable for $H^{2}(G)$, we see that  $E$ is a
continum. So it follows that there exists a function, say $h$, in
$H^{2}(G)$ that has essential singularity points contained in $E$.

On the other hand, for each $f\in N$, there is a sequence $\{f_{n}\}
$ in $Rat(K)$ such that $\lim f_{n} = f$ in $H^{2}(G)$. So this
implies that $f_{n} \rightarrow I^{-1}(f)$ and hence $\{f_{n} \}$
converges to $I^{-1}(f)$ uniformly on the compact subsets of $W$.
Consequently, we see that $f $ extends to $W$ analyticly. This
contradicts to the fact that $h$ has essential singularities on $E$.
Hence, $G$ is simply connected.

   By our hypothesis, there is a function $u\in H^{\infty}(G)$ such that
$N= \vee \{u H^{2}(G)\}$.  So we see that  $u  \in N$. Now,  let
$u=vh$, where $\psi$ is the Riemann map from $G$ onto $D$,
$v\circ\psi^{-1}$ is inner and  $h\circ\psi^{-1}$ is outer in
$H^{2}(D)$. Then  we have
\[N=v \vee \{hf: f\in H^{2}(G)\}.
\]
Since $h$ is outer, $ 1= lim_{n\rightarrow \infty} v\alpha_{n}$ for
some sequence $\{\alpha_{n}\} \in H^{2}(G)$.
 Thus,  $\psi =lim_{n\rightarrow \infty}v\alpha_{n}\psi\in N$.  Therefore, there is a sequence of
$\{r_{n}\}\in R(K)$ such that $\lim r_{n} = \psi$ in $H^{2}(G)$. But
this implies that there is a subsequence $\{r_{n_{j}}\}$ such that
 $\lim r_{n_{j}}\circ\psi^{-1}= z$ pointwise almost everywhere
on $\partial D$. Hence $\psi^{-1} $ is 1-1 a.e. $[m]]$ on $\partial
D$ and thus it follows by the definition that $G$ is a nicely
connected domain.

Now, if $G$ was not perfectly connected, then it would follows from
Theorem~\ref{t:out}
 that there is an outer function  $g\in H^{2}(D)$
such that $\{p(g\circ\psi): p\in \cal P\}$ is not dense in
$H^{2}(G)$. Set $M= \vee \{p (g\circ\psi): p\in  \cal P\}$.
 Then $M\in Lat( M_{z})$ and $M$ is nontrivial.
Again, by our hypothesis, there is a function  $u\in H^{2}(G)$ such
that $M= \vee\{u H^{2}(G) \}$. Let $u=vh$, where $v\circ\psi^{-1}$
is inner and $h\circ\psi^{-1}$ is outer in $H^{2}(D)$. Then there
exists a sequence $\{f_{n}\}$ in $ H^{2}(G)$ such that $g\circ\psi=
\lim_{n} v h  f_{n}\hspace{.02in} \mbox{ in} \hspace{.02in}
H^{2}(G).  $ This implies that $g \in v\circ \psi^{-1} H^{2}(D)$.
 Applying Beuring's theorem, we conclude that $v\circ\psi^{-1}$
 must be a constant.
 Consequently, we get that
$M = \vee \{ g H^{2}(G)\} = H^{2}(G)$, contradicting the fact that
$M$ is nontrivial.  Therefore, $G$ is a perfectly connected domain.

Sufficiency.

Let $M$ be a closed subspace of $H^{2}(G)$ such that $z M\subset M$.
 Then
$\psi^{-1} F \subset F,\hspace{.02inin}\mbox{ where }\hspace{.02in}
F=\{f\circ\psi^{-1}: f\in M\}. $
 We claim that $ z F\subset F$.
 For this fix a function $h\in F$. Let $g \in F^{\perp}$. Then
$\int h \overline g dm = 0$ and thus
\[\int p\circ\psi^{-1}h \overline g dm =0\hspace{.1in}\mbox{ for each }
\hspace{.1in}p\in \cal P.
\]
 Since $G$ is perfectly connected,
 $\psi$ is in the weak-star closure of $\cal P$. So it follows that
 $z$ belongs to the weak-star closure of $\{p\circ \psi^{-1}: p\in \cal P \}$.
This  implies that $\int zh \overline g dm = 0$. By the Hahn-Banach
theorem, we conclude that $zh \in F$ for
 each $h\in F$.
 Now applying the Beurling's theorem, we have that $F= uH^{2}(D)$
 for an inner function $u$ and therefore, we conclude that
 $M= u\circ\psi H^{2}(G)$.

\end{proof}

The following is our main theorem.

\begin{thm} \label{t:main}
Let $G$ be a bounded open subset such that no point in the
boundaries of the components of $G$ is removable for $H^{2}(G)$.  In
order that for each $M\in Lat( M_{z})$  there exist $u\in
H^{\infty}(G)$ such that $ M = \vee\{u H^{2}(G)\}$  it is necessary
and sufficient that the following hold:

1) Each component of $G$ is a perfectly connected domain.

2)  The harmonic measures of the components of $G$ are mutually
singular.

3) 
  The set of polynomials is weak-star dense in $ H^{\infty}(G)$.

\noindent Moreover, if $G$ satisfies these conditions,
 then  every $M\in Lat( M_{z})$ is of the form $u H^{2}(G)$,
where 
the restriction of $u$ to each of the components of $G$
  is either an inner function or zero.
\end{thm}

\begin{rem}
1) and 2) together insures that $H^{\infty}(G)$ can be embeded in
$L^{\infty}(\omega)$, where $\omega$ is a harmonic measure of $G$.
So $H^{\infty}(G)$ has a weak-star topology. Without 1) and 2), 3)
may not make sense. However, 1) and 2) do not imply 3). For example,
Let $a_{1}=(1,1)$, $a_{2}= (1,-1)$,
 $a_{3}=(1,i)$ and $a_{4}=(1,\hspace{.02in}-i)$; and let $G_{i}$
be the open disks that has radius 1 and is centered at $a_{i}$,
$i=1,2,3,4$, respectively.  Set $G=\cup_{i=1}^{4} G_{i}$ and set
$\omega=\sum_{i=1}^{4} \omega_{i}$, where $\omega_{i}$ is the
harmonic measure evaluated at $a_{i}$. Then $G$ satisfies conditions
1) and 2).

Let $W$ be the bounded component of the interior of the complement
of $G$ and let $ \mu=\omega|\partial W$. Then $\mu$ is a multiple of
the arclength on $\partial W$. So it follows that
\[W=\nabla P^{2}(\mu) \subset \nabla P^{2}(\omega).
\]
Since $W\cap G =\emptyset$, it follows by Proposition~\ref{p:dens}
that $P^{2}(\omega)\neq H^{2}(G)$. Consequently,
  $P^{\infty}(\omega) \neq H^{\infty}(G)$ (applying  Theorem~\ref{t:out} by
taking $g=1$).
\end{rem}

\begin{proof}
Necessity. Suppose that every $M\in Lat(M_{z})$ is of the form $\vee
\{ uH^{\infty}(G)\}$ for some $u\in H^{2}(G)$. Let $G_{1}, G_{2},
... G_{n}, ...$ be all of the components of $G$ and let $\omega_{n}$
be the harmonic measure for $G_{n}$ so that for each $r\in
Rat(\overline G_{n})$, $\|r\|_{H^{2}(G_{n})} =
\|r\|_{L^{2}(\omega_{n})}$
 for  $n=1,2, ...$  Fix  such an integer $n$ and  let us
 consider $H^{2}(G_{n})$.

Suppose that $N$ is a closed subspace of $H^{2}(G_{n})$ such that $z
H^{2}(G_{n}) \subset H^{2}(G_{n})$. Extend each $f\in N$ to be a
function in $H^{2}(G)$ by letting its value to be zero off $G_{n}$.
 Then $\chi_{G_{n}} N$ would be a closed invariant
subspace for the shift on $H^{\infty}(G)$. Thus, by our hypothesis
$\chi_{G_{n}}  N = \vee\{ u_{n} H^{2}(G)\} $ for some $u_{n} \in
H^{2}(G)$.   Therefore,
\begin{align*} N & 
 = \chi_{G_{n}}\vee\{ \chi_{G_{n}}{u_{n}} H^{2}(G)\} \\&
 =\vee \{(\chi_{G_{n}}u_{n}) H^{2}(G_{n})\}.
\end{align*}
Since $(\chi_{G_{n}}u_{n}) $ is in $H^{2}(G_{n})$, it follows by
Lemma~\ref{l:main} that $G_{n}$ must be a perfectly connected
domain.

Now set $\omega = \sum_{n=1}^{\infty} \frac{1}{2^{n}} \omega_{n}$.
Then $\omega$ is a harmonic measure for $G$.
Choose a compact subset  $K$  that contains $K$ such that the
complement of $K$ has finitely many components. Let $W= \nabla
R^{2}(K,\omega)$. We want to show that $W=G$. For this let $I$ be
the linear operator from $R^{2}(K,\omega)$ to $H^{2}(G)$
 induced by $I(f) = f$ for each $f\in Rat(K)$. Then,
 for every $f\in Rat(K)$
 \[
 \int |f|^{2}d \omega  = \sum_{n=1} \int |f|^{2}d (\frac{1}{2^{n}}\omega_{n} )
= \sum_{n=1} \frac{1}{2^{n}} \|f\|^{2}_{H^{2}(G_{n})} =
\|f\|^{2}_{H^{2}(G)}. \hspace{.1in}
\]
 So $I$ is an isometry. Clearly, $I[R^{2}(K,\omega)]$ is a closed
invariant subspace for the shift on $H^{2}(G)$. By our hypothesis,
there exists a function $u\in H^{\infty}(G)$ such that
\[I[R^{2}(K,\omega)]= \vee\{ u H^{2}(G)\}.
\]
 For each $n$, since $\chi_{G_{n}} \in H^{2}(G)$, it follows that $u\chi_{G_{n}}
\in I[R^{2}(K,\omega)]$. This implies that $u\chi_{G_{n}} \in
H^{\infty}(W)$. Let $U$ be the component of $W$ that contains
$G_{n}$. Then,  we have either $G_{n} = U$ or $u\chi_{G_{n}} $ is
identically zero on $U$. But the later implies that
$I[R^{2}(K,\omega)] = \vee \{uH^{2}(G-G_{n})\}$ and  this is clearly
impossible.
Hence we have that $U=G_{n}$.

Let $\{W_{i}\}_{i=1}^{\infty}$ be the collection of the components
of $W$. By Theorem~\ref{t:0} ,  there exists a $Borel$ partition $\{
\Delta_{i}\}_{i=0}^{\infty} $ of the support of $\omega$ such that
 \[ R^{2}(K,\omega) = R^{2}( \overline W_{1},\omega|\Delta_{1})
... \oplus R^{2}( \overline W_{i},\omega | \Delta_{i}) \oplus ...
\]
Moreover, $\overline{W}_{i} \supset  \Delta_{i} $ and
$(\omega|\Delta_{i})|\partial W_{i} \ll \omega_{W_{i}}$, where
$\omega_{W_{i}}$ is a harmonic measure for $W_{n}$; the map
$f\rightarrow \hat f $ is an isomorphism and weak-star homeomorphism
from $R^{2}(\overline W_{i}, \omega|\Delta_{i}) \cap
L^{\infty}(\omega|\Delta_{i})$ onto $H^{\infty}(W_{i})$.

 Since $\chi_{\Delta_{i}}\in R^{2}(K, \omega)$
for $i \geq 1$ and
 since $\{\Delta_{i}\}_{i=1}^{\infty}$ are pairwise disjoint,
it follows from  the bounded convergence theorem that $ 1=\lim
\sum_{i} \chi_{\Delta{i}} \hspace{.021in}\mbox{ in} \hspace{.02in}
R^{2}(K,\omega). $ Note that $\hat {f} = I (f)\hspace{.1in}\mbox{
for every} \hspace{.1in} f \in Rat(K). $
 Choose a sequence $\{f_{n}\} $ in $Rat(K)$ so that $f_{n} \rightarrow
\chi_{\Delta_{i}}$ in $L^{2}(\omega)$. Then $f_{n} \rightarrow
\widehat{\chi_{\Delta{i}}}$ uniformly compact subset of $W$.  Thus,
it follows that
\[I(\Delta_{i}) = \widehat{\Delta_{i}} = \chi_{W_{i}}
\hspace{.1in} \mbox {for each} \hspace{.1in}i.
\]
But $I(\Delta_{i}) \in H^{\infty}(G)$ and thus we conclude that
$\chi_{W_{i}} $ belongs to $H^{2}(G)$ for each $i$. This means that
$ \chi_{W_{i}}$ must be equal  to $\chi_{G_{j}}$ for some j.
Consequently, $W_{i}=G_{j}$. Therefore, it follows that $W = G$.

Now, since $\cup W_{i} = W = \cup G_{i} $ and and since each $G_{i}$
is a component of $W$, rearranging the indexes if necessary, we may
assume that $W_{i} = G_{i}$ for each $i$. From the definition of
$\omega$, it is clear that
\[\omega_{i}=\omega_{G_{i}} \ll \omega|\partial G_{i}.
\]
On the other hand, according to Theorem~\ref{t:0}, $\omega|\partial
G_{i} \ll \omega_{G_{i}}$. Hence, we have that $[\omega_{i}] =
[\omega|\partial G_{i}]\hspace{.021in}\mbox{ for each }
\hspace{.02in}i\geq 1. $ Since $\Delta_{i} \subset \partial G_{i}$
and since $\omega|\overline \Delta_{i} $ and $ \omega|\overline
\Delta_{j}$ are mutually singular,
  it follows that $\omega_{i}$ and $\omega_{j}$ are
mutually singular if $i\neq j$. Therefore, we conclude that the
harmonic measures of the components of $G$ are mutually singular.

Lastly, we want to show $P^{\infty}(G)=H^{\infty}(G)$. Suppose the
contrary. It follows from Theorem~\ref{t:out} that there exist an
outer function $g$ in $H^{2}(G)$ such that $\cal P g$ is not desne
in $H^{2}(G)$. Set $X= \vee\{pg: p\in \cal P\}$. Then $X$ is a
nontrivial invariant subspace of $M_{z}$. So there is  $u \in
H^{\infty}(G)$ such that $X= \vee \{u H^{2}(G)\}$. Let $u=vh$ be the
inner and outer factorization, then  $X=v H^{2}(G)$. Let $\{f_{n}\}
\subset H^{2}(G)$ be a sequence such that $\lim vf_{n}=g$ in
$H^{2}(G)$. Thus, we have $\lim g\chi_{\partial G_{i}} =v
f\chi_{\partial G_{i}}$ for each $i$. Since $g\chi_{\partial G_{i}}$
is outer in $H^{2}(G_{i})$, it follows by the Buerling's theorem
that $v\chi_{\partial G_{i}}$ is a constant. Consequently, we get
that $M= H^{2}(G)$.  This contradiction shows that $\cal P$ must be
weak-star dense in $P^{\infty}(G)=H^{\infty}(G)$.

Sufficiency.

Suppose $G$ is such that each of its components is perfectly
connected, the harmonic measures of its components are mutually
singular and $P^{\infty}(\omega)= H^{\infty}(G)$.
 By  Proposition~\ref{p:dens}, we have $P^{2}(\omega) = H^{2}(G)$.
 In particular, $\chi_{\partial G_{i}} \in P^{2}(\omega)$
for each $i \geq 1$. Thus, we have
\begin{align*}
P^{2}(\omega) &=P^{2}(\omega|\partial G_{1})\oplus ... \oplus
P^{2}(\omega|\partial G_{i})\oplus ... \\& =  H^{2}(G_{1})\oplus ...
\oplus H^{2}(G_{i}) \oplus ... \\& = H^{2}(G).
\end{align*}
Suppose that  $M$ is a closed $z$ invariant subspace of $H^{2}(G)$.
Pick a function $a\in M$ and set $M_{a}= \vee\{pa: p\cal P\}$. Let
$a=uh$ be the inner and outer factorization. Then it follows by
Theorem~\ref{t:out} that $M_{a} = uH^{2}(G)$. Since $h\chi_{G_{i}}
\in H^{2}(G)$, we see that $m\chi_{G_{i}} = uv\chi_{G_{i}} \in M$
for each $i \geq $. Set $M_{i} = M\chi_{G_{i}}$. Then we have
\begin{align*}
M &\subset M_{1}+ ... + M_{i} + ... \subset \overline {M_{1}}+ ... +
\overline {M_{i}}+ ... \\& \subset \overline {M_{1}}\oplus ...
\oplus \overline {M_{i}}\oplus ... \\& \subset M.
\end{align*}
Hence, $M_{i}$ is a closed subspace for each $i\geq 1$. Evidently,
$z M_{i}= \chi_{G_{i}} zM \subset M_{i}$.
 Applying Lemma~\ref{l:main},
we have that $\chi_{G_{n}} M = u_{n} H^{2}(G_{n})$, where $u_{n}=0$
if $M=\{0\}$ and  otherwise  $u_{n} \in H^{\infty}(G_{n})$ with
$|u_{n}|= 1$ a.e. $[\omega_{n}]$ on $\partial G_{n}$.
Extend $u$ to be a function in $H^{2}(G)$ by defining its value to
be zero off $\overline G_{n}$. Then, for integers $n\geq m \geq 1$,
we have that
\begin{align*}
 \|\sum_{i=1}^{n} u_{i} - \sum_{i=1}^{m} u_{i} \|_{H^{2}(G)} &
=(\sum_{i=m}^{n}
\frac{1}{2^{i}}\|u_{i}\|_{H^{2}(G_{i})})^{\frac{1}{2}} \\& \leq
(\sum_{i=m}^{n} \frac{1}{2^{i}})^{\frac{1}{2}} \rightarrow 0
\hspace{.27in} \mbox{as} \hspace{.04in}m,\hspace{.02in}n \rightarrow
\infty.
\end{align*}
Thus, $u=\lim_{n} \sum_{i=1}^{n} u_{i}$ exists. Moreover,
\begin{align*}
M &
= u_{1}H^{2}(G_{1}) \oplus ... \oplus u_{n}H^{2}(G_{n}) \oplus ...
\\& = uH^{2}(G_{1}) \oplus ... \oplus vH^{2}(G_{n}) \oplus ...  \\&
= u H^{2}(G).  \end{align*} Clearly, we have either $u=1$ or $u=0$
on $\partial G_{n}$ for each $n$.


\end{proof}


\begin{thebibliography}{0000}
\bibitem{ars}A. Aleman, S. Richter, C. Sundberg,
{\em Beurling's theorem for the Bergman space}, Acta Math. 177
(1996), 275-310.
\bibitem{abfp}C. Apostol, H. Bercovici, C. Foias, C. Pearcy,
{\em Invariant subspaces, dilation theory, and the structure of the
predual of a dual algebra}, J. Functional Analysis, 63 (1985),
369-404.
\bibitem{axler}  S. Axler,
{\em Harmonic functions from a complex analysis viewpoint}, American
Math. Monthly 93 (1986), 246-258.

\bibitem{beur} A. Beurling,
{\em  On two problems concerning linear transformations in Hilbert
spaces}, Acta Math., Vol 81 (1949), 239-255.
\bibitem{davi} A. Davie,
{\em Dirichlet algebras of analytic functions}, J. Functional
Analysis 6(1967), 348-356.
\bibitem{dure} P. Duren,
{\em Theory of $ H^{2}$ spaces}, Academic Press, New York, 1970.
\bibitem{hoff} K. Hoffman,
{\em Banach spaces of analytic functions}, Prentice-Hall, Englewood
Cliffs, N.J., 1962.
\bibitem{ga} T. Gamelin,
{\em Uniform Algebras}, Prentice Hall, Englewood Cliffs, N.J., 1969.
\bibitem{gg} T. Gamelin, J. Garnett,
{\em Constructive techniques in rational approximation}, Trans.
Amer. Math. Soc. 143 (1969), 187-200.
\bibitem{glic} I. Glicksburg,
{\em The abstract F. and M. Riese theorem, } J. Functional Analysis
1 (1967), 109-122.
\bibitem{hrs}H. Hedenhalm, S. Richter, K. Seip,
{\em Interpolating sequence and invariant subspaces of given index
in the Bergman spaces}, J. Reine Angew. Math. 477 (1996) 13-30.
\bibitem{mc}J. McCarthy, {\em Analytic structures for subnormal
operators}, Integral Equation and Operator Theorey 13(1990),
251-270.
\bibitem{dens}  Z. Qiu,
{\em Density of polynomials}, Houston Journal of Mathematics, Vol.
21, Number 1 (1995), 109-118.
 27 (1997).
\bibitem{quas}Z. Qiu, {\em On quasisimilarity of subnormal operators},
Science in China, Series A Math. 2007 Vol. 50, No. 3, 305-312.
\bibitem{}Z. Qiu, {\em Carleson measure and polynomial
approximation},  Chinese Annals of Mathematics, Series A, 2007 Vol.
28, No. 2, 159-167.
\bibitem{appr}  Z. Qiu,
{\em  The structure of the closure of the rational functions in
$L^{2}(\mu)$}, to appear. \bibitem{dens}Z. Qiu {\em Density of
rational functions in Hardy and Bergman spaces}, Indian Journal of
Pure and Applied Mathematics
\bibitem{gen} D. Sarason,
{\em Weak-star generator of $H^{\infty}$}, Pacific J. Math., 17
(1966), 519-528.
\bibitem{ord} D. Sarason,
{\em On order of a simply connected domain}, Michigan Math. J., 15
(1968),
 129-133.
\bibitem{jm} J. Thomson,
{\em Approximation in the mean by polynomials}, Annals of
Mathematics 133 (1991), 477-507.
\end{thebibliography}
\end{document}